\documentclass{my_aims}

\usepackage{amsmath}
\usepackage{paralist}
\usepackage{graphics}
\usepackage{epsfig}
\usepackage[colorlinks=true]{hyperref}
\usepackage{bm}
\usepackage{subfigure}

\hypersetup{urlcolor=blue, citecolor=red}
\usepackage[smallfootnotes=false]{hypdvips}
\newtheorem{theorem}{Theorem}[section]
\newtheorem{corollary}{Corollary}

\theoremstyle{definition}

\newtheorem{remark}{Remark}
\newtheorem{example}{Example}

\def\a{\alpha}

\def\l{\lambda}

\def\RIT{{_tI_T^{1-\a}}}
\def\LCD{{^C_aD_t^\a}}
\def\RDT{{_tD_T^\a}}
\def\RD{{_tD_b^\a}}
\def\LDa{{_aD_t^\a}}

\renewcommand{\subjclassname}{\textup{2010} Mathematics Subject Classification}

\title[Fractional order optimal control problems]{Fractional Order Optimal Control Problems\\
with Free Terminal Time}

\author[S. Pooseh, R. Almeida and D. F. M. Torres]{}



\email{spooseh@ua.pt}
\email{ricardo.almeida@ua.pt}
\email{delfim@ua.pt}

\thanks{Part of first author's Ph.D., which is carried out
at the University of Aveiro under the Doctoral Program
in Mathematics and Applications (PDMA)
of Universities of Aveiro and Minho.}

\begin{document}

\maketitle

\centerline{\scshape Shakoor Pooseh, Ricardo Almeida and Delfim F. M. Torres}
\medskip
{\footnotesize
 \centerline{CIDMA --- Center for Research and Development in Mathematics and Applications}
 \centerline{Department of Mathematics, University of Aveiro, 3810-193 Aveiro, Portugal}
}

\bigskip


\begin{abstract}
We consider fractional order optimal control problems in which
the dynamic control system involves integer
and fractional order derivatives and the terminal time is free.
Necessary conditions for a state/control/terminal-time triplet
to be optimal are obtained. Situations with constraints present
at the end time are also considered.
Under appropriate assumptions, it is shown that the
obtained necessary optimality conditions become sufficient.
Numerical methods to solve the problems are presented,
and some computational simulations are discussed in detail.
\end{abstract}


\section{Introduction}

Fractional calculus generalizes the standard integral
and differential calculus to arbitrary order.
If the order of the fractional derivative operator is an integer $m$,
then we recover an $m$-fold integral when $m$ is negative,
and the classical derivative of order $m$ when $m$ is positive.
Let $x:[a,b]\to\mathbb{R}$ be a function, and let $\alpha$ be
a real number representing the order of the integral/derivative of $x$.
The most common fractional operators considered in the literature
are the Riemann--Liouville fractional integrals (RLFI) and derivatives (RLFD),
and the Caputo fractional derivatives (CFD):
\begin{align}
{_aI_t^{\alpha}}x(t)
&= \displaystyle\frac{1}{\Gamma(\alpha)}\int_a^t (t-\tau)^{\alpha-1}x(\tau) d\tau,
\tag{left  RLFI}\\
{_tI_b^{\alpha}}x(t)
&= \displaystyle\frac{1}{\Gamma(\alpha)}\int_t^b (\tau-t)^{\alpha-1}x(\tau) d\tau,
\tag{right  RLFI}\\
{_aD_t^{\alpha}}x(t)
&= \displaystyle\frac{1}{\Gamma(n-\alpha)}\frac{d^n}{dt^n}
\int_a^t (t-\tau)^{n-\alpha-1}x(\tau) d\tau,
\tag{left  RLFD}\\
{_tD_b^{\alpha}}x(t)
&= \displaystyle\frac{(-1)^n}{\Gamma(n-\alpha)}\frac{d^n}{dt^n}
\int_t^b (\tau-t)^{n-\alpha-1}x(\tau) d\tau,
\tag{right  RLFD}\\
{_a^CD_t^{\alpha}}x(t)
&= \displaystyle\frac{1}{\Gamma(n-\alpha)}
\int_a^t (t-\tau)^{n-\alpha-1}x^{(n)}(\tau) d\tau,
\tag{left  CFD}\\
{_t^CD_b^{\alpha}}x(t)
&= \displaystyle\frac{(-1)^n}{\Gamma(n-\alpha)}
\int_t^b (\tau-t)^{n-\alpha-1}x^{(n)}(\tau) d\tau,
\tag{right CFD}
\end{align}
where $n$ stands for $\lfloor\alpha\rfloor+1$
($n$ is the smallest integer larger than $\alpha$),
and $\Gamma$ is the gamma function, that is,
$$
\Gamma(z)=\int_0^\infty t^{z-1}e^{-t}\, dt, \quad z>0.
$$
If $x$ and $x^{(i)}$, $i=1,\ldots,n-1$,
vanish at $t=a$, then ${_aD_t^{\alpha}}x(t)={_a^CD_t^{\alpha}}x(t)$,
and if they vanish at $t=b$, then ${_tD_b^{\alpha}}x(t)={_t^CD_b^{\alpha}}x(t)$.
The reader interested in the theory and applications
of fractional calculus is referred to the books \cite{Kilbas,Miller}.
For the state of the art on fractional order optimal control,
see the recent book \cite{book:FCV}.

For numerical purposes, it is sometimes useful to approximate the fractional
operators as sums of integer order (standard) derivatives. The most common procedure
is to replace the Riemann--Liouville fractional derivative by the series
$$
{_aD_t^{\alpha}} x(t)=\sum_{k=0}^\infty \binom{\a}{k}\frac{(t-a)^{k-\a}}{\Gamma(k+1-\a)}x^{(k)}(t),
\quad \text{ where } \binom{\a}{k}=\frac{(-1)^{k-1}\a\Gamma(k-\a)}{\Gamma(1-\a)\Gamma(k+1)}
$$
(cf. \cite{Miller}). This formula can be easily deduced by applying
integration-by-parts to the integral
$$
\frac{1}{\Gamma(n-\alpha)}\frac{d^n}{dt^n}\int_a^t (t-\tau)^{n-\alpha-1}x(\tau) d\tau.
$$
The main drawback is that to obtain good accuracy,
we need higher-order derivatives of $x$, which restricts
the class of admissible functions that one can consider.
Recently, in \cite{Atan2}, a new expansion formula was obtained,
with the advantage that we only need the first derivative:
$$
\LDa x(t)=A(\a)(t-a)^{-\a}x(t)+B(\a)(t-a)^{1-\a}\dot{x}(t)
-\sum_{p=2}^{\infty}C(\a,p)(t-a)^{1-p-\a}V_p(t),
$$
where $V_p(t)$ is defined as the solution of the system
$$
\begin{cases}
\dot{V}_p(t)=(1-p)(t-a)^{p-2}x(t),\\
V_p(a)=0,
\end{cases}
$$
$p=2,3,\ldots$, and the coefficients $A,B$ and $C$
are given by the formulas
\begin{equation*}
\begin{split}
A(\a)&=\displaystyle\frac{1}{\Gamma(1-\a)}
\left[1+\sum_{p=2}^{\infty}\frac{\Gamma(p-1+\a)}{\Gamma(\a)(p-1)!}\right],\\[0.3cm]
B(\a)&=\displaystyle\frac{1}{\Gamma(2-\a)}
\left[1+\sum_{p=1}^{\infty}\frac{\Gamma(p-1+\a)}{\Gamma(\a-1)p!}\right],\\[0.3cm]
C(\a,p)&=\displaystyle\frac{1}{\Gamma(2-\a)\Gamma(\a-1)}\frac{\Gamma(p-1+\a)}{(p-1)!}.
\end{split}
\end{equation*}
For computational purposes, we truncate the sum and consider the finite expansion
\begin{multline}
\label{expanMom}
\LDa x(t)\simeq A(\a,N)(t-a)^{-\a}x(t)+B(\a,N)(t-a)^{1-\a}\dot{x}(t)\\
-\sum_{p=2}^N C(\a,p)(t-a)^{1-p-\a}V_p(t),
\end{multline}
where $A(\a,N)$ and $B(\a,N)$ are now defined by
\begin{equation*}
\begin{split}
A(\a,N)&=\displaystyle\frac{1}{\Gamma(1-\a)}
\left[1+\sum_{p=2}^N\frac{\Gamma(p-1+\a)}{\Gamma(\a)(p-1)!}\right],\\
B(\a,N)&=\displaystyle\frac{1}{\Gamma(2-\a)}
\left[1+\sum_{p=1}^N\frac{\Gamma(p-1+\a)}{\Gamma(\a-1)p!}\right].
\end{split}
\end{equation*}
We refer to \cite{Pooseh2} where expansion formulas with higher-order derivatives
are obtained, and an error estimation is proven. In \cite{Pooseh3} analogous results
are proven for the Riemann--Liouville fractional integral and in \cite{Pooseh1}
for Hadamard fractional operators.
For the right Riemann--Liouville fractional derivative, we have
\begin{multline}
\label{expanMomR}
\RD x(t)\simeq A(\a,N)(b-t)^{-\a}x(t)-B(\a,N)(b-t)^{1-\a}\dot{x}(t)\\
+\sum_{p=2}^NC(\a,p)(b-t)^{1-p-\a}W_p(t),
\end{multline}
where $W_p$ is the solution of the differential equation
$$
\begin{cases}
\dot{W}_p(t)=-(1-p)(b-t)^{p-2}x(t),\\
W_p(b)=0.
\end{cases}
$$
To approximate the Caputo fractional derivatives, we may use the formulas
\begin{equation*}
{_a^CD_t^\alpha}x(t)={_aD_t^\alpha}x(t)
-\sum_{k=0}^{n-1}\frac{x^{(k)}(a)}{\Gamma(k-\alpha+1)}(t-a)^{k-\alpha}
\end{equation*}
and
$$
{_t^CD_b^\alpha}x(t)={_tD_b^\alpha}x(t)
-\sum_{k=0}^{n-1}\frac{x^{(k)}(b)}{\Gamma(k-\alpha+1)}(b-t)^{k-\alpha},
$$
that relate Caputo and Riemann--Liouville fractional operators,
and then use \eqref{expanMom} and \eqref{expanMomR}.
In this article we are interested in investigating such ideas
in the context of fractional order optimal control.
An important tool is the integration by parts
formula for Caputo fractional derivatives,
which is stated in the following theorem.

\begin{theorem}[cf., e.g., \cite{Agrawal10}]
Let $\alpha\in(0,1)$, and $x,y:[a,b]\to\mathbb{R}$ be two functions of class $C^1$.
Then the following integration by parts formula holds:
\begin{equation*}
\int_{a}^{b}y(t) {_a^C D_t^\alpha}x(t)dt
=\left[{_t I_b^{1-\alpha}}y(t) x(t)\right]_a^b
+ \int_a^b x(t) {_t D_b^\alpha}y(t)dt.
\end{equation*}
\end{theorem}

The text is organized as follows. In Section~\ref{SecNecCond}
we formulate the optimal control problem under consideration
and deduce necessary optimality conditions for it
(Theorem~\ref{Opt:MainTheo}). Another approach consists
of using the approximation methods mentioned above,
thereby converting the original problem into
a classical optimal control problem that
can be solved by standard computational techniques
(Section~\ref{sec:2.2}). A generalization of the obtained
results is considered in Section~\ref{sec:3},
where the lower bound of the cost integral
is different from the lower bound of the fractional derivative
(Theorem~\ref{thm:mt:gen}). Under some additional assumptions,
the necessary optimality conditions
are also sufficient (Theorem~\ref{thm:suff:cond}).
We end with Section~\ref{sec:5} containing numerical computations.


\section{Necessary optimality conditions}
\label{SecNecCond}

Let $\a\in(0,1)$, $a\in\mathbb{R}$, $L$ and $f$ be two differentiable functions
with domain $[a,+\infty)\times \mathbb{R}^2$, and
$\phi : [a,+\infty)\times \mathbb{R} \rightarrow \mathbb{R}$
a differentiable function. The fundamental problem
is stated in the following way:
\begin{equation}
\label{Opt:func}
\mathrm{minimize} \quad
J(x,u,T)=\int_a^T L(t,x(t),u(t))\,dt+\phi(T,x(T))
\end{equation}
subject to the control system
\begin{equation}
\label{Opt:dynamic}
M \dot{x}(t) + N\LCD x(t) = f(t,x(t),u(t))
\end{equation}
and the initial boundary condition
\begin{equation}
\label{Opt:bound}
x(a)=x_a
\end{equation}
with $(M,N)\not=(0,0)$ and $x_a$ a fixed real number.
Our goal is to generalize previous works
on fractional optimal control problems
by considering the end time $T$ free and the dynamic control system \eqref{Opt:dynamic}
involving integer and fractional order derivatives. For convenience,
we consider the one-dimensional case. However, using similar techniques as the ones given here,
the results can be easily extended to problems with multiple states and multiple controls.
Later we consider the cases $T$ and/or $x(T)$ fixed.
Here, $T$ is a variable number with $a< T<\infty$.
Thus, we are interested not only on the optimal trajectory
$x$ and optimal control function $u$,
but also on the corresponding time $T$ for which
the functional $J$ attains its minimum value.
We assume that the state variable $x$ is differentiable and that
the control $u$ is piecewise continuous. When $N=0$ we obtain a classical
optimal control problem; the case $M=0$ with fixed $T$ has already been
studied for different types of fractional order derivatives
(see, e.g., \cite{Agrawal1,Agrawal2,Agrawal3,FreTorres,FT,Tricaud,Tricaud2}).
In \cite{Jelicic} a special type of the proposed problem is also studied
for fixed $T$.

\begin{remark}
In this paper the terminal time $T$ is a free decision variable
and, \textit{a priori}, no constraints are imposed.
For future research, one may wish to consider a class of
fractional optimal control problems in which the terminal time
is governed by a stopping condition. Such problems were recently investigated,
within the classical (integer order) framework, in \cite{Lin1,Lin2}.
\end{remark}


\subsection{Fractional necessary conditions}

To deduce necessary optimality conditions that an optimal triplet
$(x,u,T)$ must satisfy, we use a Lagrange multiplier
to adjoin the dynamic constraint \eqref{Opt:dynamic}
to the performance functional \eqref{Opt:func}. To start,
we define the Hamiltonian function $H$ by
\begin{equation}
\label{def:Hamiltonian}
H(t,x,u,\l)=L(t,x,u)+\l f(t,x,u),
\end{equation}
where $\l$ is a Lagrange multiplier,
so that we can rewrite the initial problem as minimizing
$$
\mathcal{J}(x,u,T,\l)=\int_a^T \left[H(t,x,u,\l)-\l(t)[M \dot{x}(t)+N\LCD x(t)]\right]\,dt+\phi(T,x(T)).
$$
Next, we consider variations of the form
$$
x+ \delta x, \quad u+\delta u, \quad T+\delta T, \quad \l+\delta\l
$$
with $\delta x(a)=0$ by the imposed boundary condition \eqref{Opt:bound}.
Using the well-known fact that the first variation of $\mathcal J$
must vanish when evaluated along a minimizer, we get
\begin{multline*}
0 = \int_a^T\Biggl(\frac{\partial H}{\partial x}\delta x+\frac{\partial H}{\partial u}\delta u
+\frac{\partial H}{\partial \l}\delta \l-\delta \l\left(M \dot{x}(t)
+N\LCD x(t)\right)\\
-\l(t)\left(M \dot{\delta x}(t)+N\LCD\delta x(t)\right)\Biggr)dt
+\delta T\bigl[H(t,x,u,\l)\\
-\l(t)\left(M \dot{x}(t)+N\LCD x(t)\right)\bigr]_{t=T}
+\frac{\partial\phi}{\partial t}(T,x(T)) \delta T
+\frac{\partial \phi}{\partial x}(T,x(T))\left(\dot{x}(T)\delta T+\delta x(T)\right)
\end{multline*}
with the partial derivatives of $H$ evaluated at $(t,x(t),u(t),\l(t))$.
Integration by parts gives the relations
$$
\int_a^T\l(t) \dot{\delta x}(t)\, dt
=-\int_a^T \delta x(t)  \dot{\l}(t)\,dt+\delta x(T)\l(T)
$$
and
$$
\int_a^T\l(t) \LCD \delta x(t)\,dt
=\int_a^T \delta x(t) \, \RDT\l(t)\,dt
+\delta x(T)[\RIT\l(t)]_{t=T}.
$$
Thus, we deduce the following formula:
\begin{multline*}
\int_a^T\left[\delta x \left(\frac{\partial H}{\partial x}
+M \dot{\l}(t) -N\RDT\l(t)\right)+\delta u \frac{\partial H}{\partial u}\right.\\
\left.+\delta \l\left( \frac{\partial H}{\partial \l}-M \dot{x}(t)-N\LCD x(t)\right)\right]dt\\
-\delta x(T)\left[M \l(t)+N\RIT\l(t)-\frac{\partial \phi}{\partial x}(t,x(t))\right]_{t=T}\\
+\delta T[H(t,x,u,\l)-\l(t)[M \dot{x}(t)+N\LCD x(t)]+\frac{\partial \phi}{\partial t}(t,x(t))
+\frac{\partial \phi}{\partial x}(t,x(t))\dot{x}(t)]_{t=T}=0.
\end{multline*}
Now, define the new variable
$$
\delta x_T=(x+\delta x)(T+\delta T)-x(T).
$$
Because $\dot{\delta x}(T)$ is arbitrary,
in particular one can consider variation functions
for which $\dot{\delta x}(T)=0$.
By Taylor's theorem,
$$
(x+\delta x)(T+\delta T)-(x+\delta x)(T)=\dot{x}(T)\delta T+O(\delta T^2),
$$
where $\displaystyle \lim_{\zeta\to0} \frac{O(\zeta)}{\zeta}$ is finite, and so
$\delta x(T)=\delta x_T-\dot{x}(T)\delta T+O(\delta T^2)$.
In conclusion, we arrive at the expression
\begin{multline*}
\delta T\left[H(t,x,u,\l)-N\l(t)\LCD x(t)+N\dot{x}(t)\RIT\l(t)
+\frac{\partial \phi}{\partial t}(t,x(t))\right]_{t=T}\\
+\int_a^T\left[\delta x \left(\frac{\partial H}{\partial x}
+M \dot{\l}(t)-N\RDT\l(t)\right)
+\delta \l\left(\frac{\partial H}{\partial \l}-M \dot{x}(t)-N\LCD x(t)\right)\right.\\
\left.+\delta u \frac{\partial H}{\partial u}\right]dt
-\delta x_T\left[M \l(t)+N\RIT\l(t)-\frac{\partial \phi}{\partial x}(t,x(t))\right]_{t=T}
+O(\delta T^2)=0.
\end{multline*}
Since the variation functions were chosen arbitrarily, the following theorem is proven.

\begin{theorem}
\label{Opt:MainTheo}
If $(x,u,T)$ is a minimizer of \eqref{Opt:func} under the dynamic constraint
\eqref{Opt:dynamic} and the boundary condition \eqref{Opt:bound},
then there exists a function $\l$ for which the triplet $(x,u,\l)$  satisfies:
\begin{itemize}
\item the \emph{Hamiltonian system}
\begin{equation}
\label{OPT:HamilSyst}
\begin{cases}
M \dot{\l}(t) - N\RDT\l(t) = - \frac{\partial H}{\partial x}(t,x(t),u(t),\l(t))\\
M \dot{x}(t)+N\LCD x(t) = \frac{\partial H}{\partial \l}(t,x(t),u(t),\l(t))
\end{cases}
\end{equation}
for all $t\in[a,T]$;

\item the \emph{stationary condition}
\begin{equation}
\label{OPT:stacionary}
\frac{\partial H}{\partial u}(t,x(t),u(t),\l(t))=0
\end{equation}
for all $t\in[a,T]$;

\item and the \emph{transversality conditions}
\begin{equation}
\label{OPT:transversality}
\begin{gathered}
\left[H(t,x(t),u(t),\l(t))-N\l(t)\LCD x(t)+N\dot{x}(t)\RIT\l(t)
+\frac{\partial \phi}{\partial t}(t,x(t))\right]_{t=T}=0,\\
\left[M \l(t) +N\RIT\l(t)-\frac{\partial \phi}{\partial x}(t,x(t))\right]_{t=T}=0;
\end{gathered}
\end{equation}
\end{itemize}
where the Hamiltonian $H$ is defined by \eqref{def:Hamiltonian}.
\end{theorem}

\begin{remark}
In standard optimal control, a free terminal time problem
can be converted into a fixed final time problem
by using the well-known transformation $s = t/T$ (see Example~\ref{exm42}).
This transformation does not work in the fractional setting.
Indeed, in standard optimal control, translating the problem from time $t$
to a new time variable $s$ is straightforward: the chain rule gives
$\frac{d x}{d s} = \frac{d x}{d t} \frac{d t}{d s}$.
For Caputo or Riemann--Liouville fractional derivatives,
the chain rule is not valid and such conversion is not possible.
\end{remark}

Some interesting special cases are obtained when restrictions
are imposed on the end time $T$ or on $x(T)$.

\begin{corollary}
\label{OPT:MainCor}
Let $(x,u)$ be a minimizer of \eqref{Opt:func} under the dynamic constraint
\eqref{Opt:dynamic} and the boundary condition \eqref{Opt:bound}.
\begin{enumerate}
\item If $T$ is fixed and $x(T)$ is free,
then Theorem~\ref{Opt:MainTheo} holds with the transversality conditions
\eqref{OPT:transversality} replaced by
$$
\left[M \l(t)+N\RIT\l(t)-\frac{\partial \phi}{\partial x}(t,x(t))\right]_{t=T}=0.
$$

\item If $x(T)$ is fixed and $T$ is free,
then Theorem~\ref{Opt:MainTheo} holds with the transversality
conditions \eqref{OPT:transversality} replaced by
$$
\left[H(t,x(t),u(t),\l(t))-N\l(t)\LCD x(t)+N\dot{x}(t)\RIT\l(t)
+\frac{\partial \phi}{\partial t}(t,x(t))\right]_{t=T}=0.
$$

\item If $T$ and $x(T)$ are both fixed,
then Theorem~\ref{Opt:MainTheo} holds
with no transversality conditions.

\item If the terminal point $x(T)$ belongs to a fixed curve, i.e.,
$x(T)=\gamma(T)$ for some differentiable curve $\gamma$, then
Theorem~\ref{Opt:MainTheo} holds with the transversality conditions
\eqref{OPT:transversality} replaced by
\begin{multline*}
\Biggl[H(t,x(t),u(t),\l(t))-N\l(t)\LCD x(t)+N\dot{x}(t)\RIT\l(t)+\frac{\partial \phi}{\partial t}(t,x(t))\\
-\dot{\gamma}(t)\left(M \l(t)+N\RIT\l(t)-\frac{\partial \phi}{\partial x}(t,x(t))\right)\Biggr]_{t=T}=0.
\end{multline*}

\item If $T$ is fixed and $x(T)\geq K$ for some fixed $K\in\mathbb{R}$,
then Theorem~\ref{Opt:MainTheo} holds with the transversality conditions
\eqref{OPT:transversality} replaced by
\begin{gather*}
\left[M \l(t)+N\RIT\l(t)-\frac{\partial \phi}{\partial x}(t,x(t))\right]_{t=T}\leq 0,\\
(x(T)-K)\left[M \l(t)+N\RIT\l(t)-\frac{\partial \phi}{\partial x}(t,x(t))\right]_{t=T}=0.
\end{gather*}

\item If $x(T)$ is fixed and $T\leq K$ for some fixed $K\in\mathbb{R}$,
then Theorem~\ref{Opt:MainTheo} holds with the transversality conditions
\eqref{OPT:transversality} replaced by
$$
\left[H(t,x(t),u(t),\l(t))-N\l(t)\LCD x(t)+N\dot{x}(t)\RIT\l(t)
+\frac{\partial \phi}{\partial t}(t,x(t))\right]_{t=T}\geq 0,
$$
\begin{multline*}
\left[H(t,x(t),u(t),\l(t))-N\l(t)\LCD x(t)+N\dot{x}(t)\RIT\l(t)
+\frac{\partial \phi}{\partial t}(t,x(t))\right]_{t=T}\\
\times (T-K) = 0.
\end{multline*}
\end{enumerate}
\end{corollary}

\begin{proof}
The first three conditions are obvious. The fourth follows from
$$
\delta x_T=\gamma(T+\delta T)-\gamma(T)=\dot{\gamma}(T)\delta T+O(\delta T^2).
$$
To prove \textit{5}, observe that we have two possible cases. If $x(T)>K$,
then $\delta x_T$ may take negative and positive values, and so we get
$$
\left[M \l(t)+N\RIT\l(t)-\frac{\partial \phi}{\partial x}(t,x(t))\right]_{t=T}=0.
$$
On the other hand, if $x(T)=K$, then $\delta x_T\geq 0$ and so
$$
\left[M \l(t)+N\RIT\l(t)-\frac{\partial \phi}{\partial x}(t,x(t))\right]_{t=T}\leq 0.
$$
The proof of the last condition is similar.
\end{proof}

Case~1 of Corollary~\ref{OPT:MainCor} was proven in \cite{FreTorres}
for $(M,N)=(0,1)$ and $\phi \equiv 0$. Moreover, if $\a=1$, then we obtain the classical
necessary optimality conditions for the standard optimal control problem (see, e.g., \cite{Chiang}):
the Hamiltonian system
$$
\begin{cases}
\dot{x}(t) = \frac{\partial H}{\partial \l}(t,x(t),u(t),\l(t)),\\
\dot{\l}(t) = - \frac{\partial H}{\partial x}(t,x(t),u(t),\l(t));
\end{cases}
$$
the stationary condition $\displaystyle \frac{\partial H}{\partial u}(t,x(t),u(t),\l(t))=0$;
and the transversality condition $\l(T)=0$.


\subsection{Approximated integer order necessary optimality conditions}
\label{sec:2.2}

Using approximation \eqref{expanMom} up to order $K$, we can transform
the original problem \eqref{Opt:func}--\eqref{Opt:bound}
into the following classical problem:
\begin{equation*}
\mathrm{minimize} \quad
\tilde{J}(x,u,T)=\int_a^TL(t,x(t),u(t))\,dt+\phi(T,x(T))
\end{equation*}
subject to
\begin{equation*}
\begin{cases}
\dot{x}(t)=\displaystyle \frac{f(t,x(t),u(t))
-NA(t-a)^{-\a}x(t)+\sum_{p=2}^KNC_p(t-a)^{1-p-\a}V_p(t)}{M+NB(t-a)^{1-\a}},\\[0.25cm]
\dot{V}_p(t)=(1-p)(t-a)^{p-2}x(t), \quad p=2,\ldots,K,
\end{cases}
\end{equation*}
and
\begin{equation}
\label{App:bound}
\begin{cases}
x(a)=x_a,\\
V_p(a)=0, \quad p=2,\ldots,K,
\end{cases}
\end{equation}
where $A=A(\a,K)$, $B=B(\a,K)$ and $C_p=C(\a,p)$
are the coefficients in the approximation \eqref{expanMom}.
Now that we are dealing with an integer order problem,
we can follow a classical procedure (see, e.g., \cite{Kirk}),
by defining the Hamiltonian $H$ by
\begin{equation*}
\begin{split}
H=L(t,x,u)&+\frac{\l_1 \left(f(t,x,u) -NA(t-a)^{-\a}x
+\sum_{p=2}^KNC_p(t-a)^{1-p-\a}V_p\right)}{M+NB(t-a)^{1-\a}}\\
&+\sum_{p=2}^K \l_p (1-p)(t-a)^{p-2}x.
\end{split}
\end{equation*}
Let $\bm{\l}=(\l_1,\l_2,\ldots,\l_K)$ and $\mathbf{x}=(x,V_2,\ldots,V_K)$.
The necessary optimality conditions
\begin{equation*}
\frac{\partial H}{\partial u}=0, \quad
\begin{cases}
\dot{\mathbf{x}}=\displaystyle \frac{\partial H}{\partial \bm{\l}},\\[0.25cm]
\dot{\bm{\l}}=-\displaystyle \frac{\partial H}{\partial \mathbf{x}},
\end{cases}
\end{equation*}
result in a two point boundary value problem.
Assume that $(T^*,\mathbf{x}^*,\bm{u}^*)$ is the optimal triplet.
In addition to the boundary conditions \eqref{App:bound},
the transversality conditions imply
$$
\left[\frac{\partial \phi}{\partial \mathbf{x}}(T^*,
\mathbf{x}^*(T))\right]^{tr}\delta \mathbf{x}_T
+\left[H(T^*, \mathbf{x}^*(T),\bm{u}^*(T),\bm{\l}^*(T))
+\frac{\partial \phi}{\partial t}(T^*, \mathbf{x}^*(T))\right]\delta T=0,
$$
where $tr$ denotes the transpose.
Because $V_p$, $p=2,\ldots,K$, are auxiliary variables
whose values $V_p(T)$, at the final time $T$, are free, we have
$$
\l_p(T)=\frac{\partial\phi}{\partial V_p}\Big |_{t=T} = 0,
\quad p=2,\ldots,K.
$$
The value of $\l_1(T)$ is determined from the value of $x(T)$.
If $x(T)$ is free, then $\l_1(T)=\frac{\partial\phi}{\partial x}|_{t=T}$.
Whenever the final time is free, a transversality condition of the form
$$
\left[H\left(t,\mathbf{x}(t),\bm{u}(t),\bm{\l}(t)\right)
-\frac{\partial\phi}{\partial t}\left(t,\mathbf{x}(t)\right)\right]_{t=T}=0
$$
completes the required set of boundary conditions.


\section{A generalization}
\label{sec:3}

The aim is now to consider a generalization of the
optimal control problem \eqref{Opt:func}--\eqref{Opt:bound}
studied in Section~\ref{SecNecCond}. Observe that the initial point $t=a$ is in fact
the initial point for two different operators: for the integral in \eqref{Opt:func}
and, secondly, for the left Caputo fractional derivative given by the dynamic constraint
\eqref{Opt:dynamic}. We now consider the case where the lower bound of the integral
of $J$ is greater than the lower bound of the fractional derivative.
The problem is stated as follows:
\begin{equation}
\label{eq:gJ}
\mathrm{minimize} \quad J(x,u,T)=\int_A^TL(t,x(t),u(t))\,dt+\phi(T,x(T))
\end{equation}
under the constraints
\begin{equation}
\label{eq:gCS}
M \dot{x}(t)+N\LCD x(t)=f(t,x(t),u(t)) \quad \mbox{and} \quad x(A)=x_A,
\end{equation}
where $(M,N)\not=(0,0)$, $x_A$ is a fixed real, and $a<A$.

\begin{remark}
We have chosen to consider the initial condition
on the initial time $A$ of the cost integral,
but the case of initial condition $x(a)$ instead of $x(A)$
can be studied using similar arguments.
Our choice seems the most natural: the interval of interest is $[A,T]$
but the fractional derivative is a nonlocal operator and has ``memory''
that goes to the past of the interval $[A,T]$ under consideration.
\end{remark}

\begin{remark}
In the theory of fractional differential equations,
the initial condition is given at $t=a$. To the best of our
knowledge there is no general theory about unicity of solutions
for problems like \eqref{eq:gCS}, where the fractional derivative
involves $x(t)$ for $a<t<A$ and the initial condition is given at $t=A$.
Unicity of solution is, however, possible. Consider, for example,
${_0^C D ^\alpha _t} x(t)=t^2$. Applying the fractional integral to both sides of equality
we get $x(t)=x(0)+ 2 t^{2+\alpha}/\Gamma(3+\alpha)$ so, knowing a value for $x(t)$,
not necessarily at $t=0$, one can determine $x(0)$ and by doing so $x(t)$.
A different approach than the one considered here
is to provide an initialization function for $t\in[a,A]$.
This initial memory approach was studied for fractional
continuous-time linear control systems in \cite{dorota:t:cap}
and \cite{dorota:t:RL}, respectively for Caputo and Riemann--Liouville derivatives.
\end{remark}

The method to obtain
the required necessary optimality conditions follows the same procedure
as the one discussed before. The first variation gives
\begin{equation*}
\begin{split}
0 = \int_A^T&\Biggl[\frac{\partial H}{\partial x}\delta x+\frac{\partial H}{\partial u}\delta u
+\frac{\partial H}{\partial \l}\delta \l-\delta \l\left(M \dot{x}(t)+N\LCD x(t)\right)\\
&-\l(t)\left(M \dot{\delta x}(t)+N\LCD\delta x(t)\right)\Biggr]dt
+\frac{\partial \phi}{\partial x}(T,x(T))\left(\dot{x}(T)\delta T+\delta x(T)\right)\\
&+\frac{\partial\phi}{\partial t}(T,x(T)) \delta T
+\delta T\left[H(t,x,u,\l)-\l(t)\left(M \dot{x}(t)+N\LCD x(t)\right)\right]_{t=T},
\end{split}
\end{equation*}
where the Hamiltonian $H$ is as in \eqref{def:Hamiltonian}.
Now, if we integrate by parts, we get
$$
\int_A^T\l(t) \dot{\delta x}(t)\, dt
=-\int_A^T \delta x(t)  \dot{\l}(t)\,dt+\delta x(T)\l(T)
$$
and
\begin{equation*}
\begin{split}
\int_A^T & \l(t) \LCD \delta x(t)\,dt
=\int_a^T \l(t) \LCD \delta x(t)\,dt-\int_a^A\l(t) \LCD \delta x(t)\,dt\\
&=\int_a^T \delta x(t) \, \RDT\l(t)\,dt+[\delta x(t) \RIT\l(t)]_{t=a}^{t=T}
-\int_a^A \delta x(t) \, {_tD^\a_A}\l(t)\,dt\\
&\quad -[\delta x(t){_tI^{1-\a}_A}\l(t)]_{t=a}^{t=A}\\
&=\int_a^A \delta x(t) [\RDT\l(t)-{_tD^\a_A}\l(t)]\,dt
+\int_A^T \delta x(t) \, \RDT\l(t)\,dt\\
&\quad +\delta x(T)[\RIT\l(t)]_{t=T}
-\delta x(a)[{_aI^{1-\a}_T}\l(a)-{_aI^{1-\a}_A}\l(a)].
\end{split}
\end{equation*}
Substituting these relations into the first variation of $J$, we conclude that
\begin{equation*}
\begin{split}
\int_A^T&\left[\left(\frac{\partial H}{\partial x}
+M \dot{\l} -N\RDT\l\right)\delta x
+\frac{\partial H}{\partial u} \delta u
+\left( \frac{\partial H}{\partial \l}-M \dot{x}-N\LCD x\right)\delta \l \right]dt\\
&-N\int_a^A \delta x [\RDT\l-{_tD^\a_A}\l]\,dt
-\delta x[M \l+N\RIT\l-\frac{\partial \phi}{\partial x}(t,x)]_{t=T}\\
&+\delta T[H(t,x,u,\l)-\l[M \dot{x}+N\LCD x]
+\frac{\partial \phi}{\partial t}(t,x)
+\frac{\partial \phi}{\partial x}(t,x)\dot{x}]_{t=T}\\
&+N\delta x(a)[{_aI^{1-\a}_T}\l(a)-{_aI^{1-\a}_A}\l(a)]=0.
\end{split}
\end{equation*}
Repeating the calculations as before, we prove the following optimality conditions.

\begin{theorem}
\label{thm:mt:gen}
If the triplet $(x,u,T)$ is an optimal solution to problem \eqref{eq:gJ}--\eqref{eq:gCS},
then there exists a function $\l$ for which the following conditions hold:
\begin{itemize}
\item the \emph{Hamiltonian system}
\begin{equation*}
\begin{cases}
M \dot{\l}(t)-N\RDT\l(t) = - \displaystyle \frac{\partial H}{\partial x}(t,x(t),u(t),\l(t))\\[0.25cm]
M \dot{x}(t)+N\LCD x(t) = \displaystyle \frac{\partial H}{\partial \l}(t,x(t),u(t),\l(t))
\end{cases}
\end{equation*}
for all $t\in[A,T]$, and
$\RDT\l(t)-{_tD^\a_A}\l(t)=0$ for all $t\in[a,A]$;

\item the \emph{stationary condition}
$$
\frac{\partial H}{\partial u}(t,x(t),u(t),\l(t))=0
$$
for all $t\in[A,T]$;

\item the \emph{transversality conditions}
\begin{gather*}
\left[H(t,x(t),u(t),\l(t))-N\l(t)\LCD x(t)+N\dot{x}(t)\RIT\l(t)
+\frac{\partial \phi}{\partial t}(t,x(t))\right]_{t=T}=0,\\
\left[M \l(t)+N\RIT\l(t)-\frac{\partial \phi}{\partial x}(t,x(t))\right]_{t=T}=0,\\
\left[{_tI^{1-\a}_T}\l(t)-{_tI^{1-\a}_A}\l(t)\right]_{t=a}=0;
\end{gather*}
\end{itemize}
with the Hamiltonian $H$ given by \eqref{def:Hamiltonian}.
\end{theorem}

\begin{remark}
If the admissible functions take fixed values at both $t=a$ and $t=A$,
then we only obtain the two transversality conditions evaluated at $t=T$.
\end{remark}


\section{Sufficient optimality conditions}
\label{sec:4}

In this section we show that, under some extra hypotheses,
the obtained necessary optimality conditions
are also sufficient.

\begin{theorem}
\label{thm:suff:cond}
Let $(\overline{x}, \overline{u}, \overline{\l})$ be a triplet satisfying conditions
\eqref{OPT:HamilSyst}--\eqref{OPT:transversality} of Theorem~\ref{Opt:MainTheo}.
Moreover, assume that
\begin{enumerate}
\item $L$ and $f$ are convex on $x$ and $u$, and $\phi$ is convex in $x$;
\item $T$ is fixed;
\item $\overline{\l}(t)\geq 0$ for all $t \in [a,T]$ or $f$ is linear in $x$ and $u$.
\end{enumerate}
Then $(\overline x,\overline u)$ is an optimal solution to problem
\eqref{Opt:func}--\eqref{Opt:bound}.
\end{theorem}

\begin{proof}
From \eqref{OPT:HamilSyst} we deduce that
$$
\frac{\partial L}{\partial x}(t,\overline x(t),\overline u(t))
=-M\dot{\overline \l}(t)+N\RDT\overline \l(t)
-\overline\l(t)\frac{\partial f}{\partial x}(t,\overline x(t),\overline u(t)).
$$
Using \eqref{OPT:stacionary},
$$
\frac{\partial L}{\partial u}(t,\overline x(t),\overline u(t))
=-\overline\l(t)\frac{\partial f}{\partial u}(t,\overline x(t),\overline u(t))
$$
and \eqref{OPT:transversality} gives
$[M\overline \lambda(t)+N{_tI_T^{1-\alpha}}\overline \lambda(t)
-\frac{\partial \phi}{\partial x}(t,\overline x(t))]_{t=T}=0$.
Let $(x,u)$ be admissible, i.e., let \eqref{Opt:dynamic}
and\eqref{Opt:bound} be satisfied for $(x,u)$. In this case,
\begin{equation*}
\begin{split}
J&(x,u)-J(\overline x,\overline u) \\
&=\int_a^T \left[L(t,x(t),u(t))-L(t,\overline x(t),\overline u(t))\right] dt
+\phi(T,x(T))-\phi(T,\overline x(T))\\
&\geq \int_a^T \left[\frac{\partial L}{\partial x}(t,\overline x(t),\overline u(t)) (x(t)-\overline x(t))
+\frac{\partial L}{\partial u}(t,\overline x(t),\overline u(t))(u(t)-\overline u(t))\right] dt\\
&\qquad +\frac{\partial \phi}{\partial x}(T,\overline x(T))(x(T)-\overline x(T))\\
&=\int_a^T \biggl[ -M\dot{\overline \l}(t)(x(t)-\overline x(t))+N\RDT\overline \l(t)(x(t)-\overline x(t))\\
&\qquad -\overline\l(t)\frac{\partial f}{\partial x}(t,\overline x(t),\overline u(t))(x(t)-\overline x(t))
-\overline\l(t)\frac{\partial f}{\partial u}(t,\overline x(t),\overline u(t))(u(t)-\overline u(t))\biggr] dt\\
&\qquad +\frac{\partial \phi}{\partial x}(T,\overline x(T))(x(T)-\overline x(T)).
\end{split}
\end{equation*}
Integrating by parts, and noting that $x(a)=\overline x(a)$, we obtain
\begin{equation*}
\begin{split}
J&(x,u)-J(\overline x,\overline u)\\
&=\int_a^T \overline \l(t)\Biggl[ M\left(\dot{x}(t)-\dot{\overline x}(t)\right)
+N\left(\LCD x(t)-\LCD \overline x(t)\right)
\end{split}
\end{equation*}
\begin{equation*}
\begin{split}
&\quad -\frac{\partial f}{\partial x}\left(t,\overline x(t),\overline u(t)\right)\left(x(t)-\overline x(t)\right)
-\frac{\partial f}{\partial u}\left(t,\overline x(t),\overline u(t)\right)\left(u(t)-\overline u(t)\right)\Biggr] dt\\
&\quad +\left[\frac{\partial \phi}{\partial x}(t,\overline x(t))
-M\overline \lambda(t)-N{_tI_T^{1-\alpha}}\overline \lambda(t)\right]_{t=T}
\left(x(T)-\overline x(T)\right)\\
&= \int_a^T \Biggl[\overline \l(t)\left[ f(t,x(t),u(t))-f\left(t,\overline x(t),\overline u(t)\right) \right]
-\overline\l(t)\frac{\partial f}{\partial x}(t,\overline x(t),\overline u(t))\left(x(t)-\overline x(t)\right)\\
&\quad -\overline\l(t)\frac{\partial f}{\partial u}(t,
\overline x(t),\overline u(t))\left(u(t)-\overline u(t)\right)\Biggr] dt\\
&\geq \int_a^T \overline \l(t) \Biggl[
\frac{\partial f}{\partial x}\left(t,\overline x(t),\overline u(t)\right)\left(x(t)-\overline x(t)\right)
+\frac{\partial f}{\partial u}\left(t,\overline x(t),\overline u(t)\right)\left(u(t)-\overline u(t)\right)\\
&\quad -\frac{\partial f}{\partial x}\left(t,\overline x(t),\overline u(t)\right)\left(x(t)-\overline x(t)\right)
-\frac{\partial f}{\partial u}\left(t,\overline x(t),\overline u(t)\right)\left(u(t)-\overline u(t)\right)\Biggr] dt\\
&= 0.
\end{split}
\end{equation*}
\end{proof}

\begin{remark}
If the functions in Theorem~\ref{thm:suff:cond} are strictly convex
instead of convex, then the minimizer is unique.
\end{remark}


\section{Numerical treatment and examples}
\label{sec:5}

Here we apply the necessary conditions of Section~\ref{SecNecCond}
to solve some test problems. Solving an optimal control problem,
analytically, is an optimistic goal and is impossible
except for simple cases. Therefore, we apply numerical and computational methods
to solve our problems. In each case we try to solve the problem either
by applying fractional necessary conditions or by approximating the problem
by a classical one and then solving the approximate problem.


\subsection{Fixed final time}
\label{sub:sec:fft}

We first solve a simple problem with fixed final time.
In this case the exact solution, i.e.,
the optimal control and the corresponding optimal trajectory,
is known, and hence we can compare it with the approximations
obtained by our numerical method.

\begin{example}
\label{exm41}
Consider the following optimal control problem:
$$
J(x,u)=\int_0^1 \left(t u(t)-(\a+2)x(t)\right)^2\,dt \longrightarrow \min
$$
subject to the control system
$$
\dot{x}(t)+{^C_0D^\a_t} x(t)=u(t)+t^2
$$
and the boundary conditions
$$
x(0)=0, \quad x(1)=\frac{2}{\Gamma(3+\a)}.
$$
The solution is given by
$$
\left(\overline x(t),\overline u(t)\right)
=\left(\frac{2t^{\a+2}}{\Gamma(\a+3)},\frac{2t^{\a+1}}{\Gamma(\a+2)}\right),
$$
because $J(x,u) \geq 0$ for all pairs $(x,u)$ and
$\overline x(0)=0$, $\overline x(1)=\frac{2}{\Gamma(3+\a)}$,
$\dot{\overline{x}}(t) = \overline{u}(t)$ and
${^C_0D^\a_t} \overline x (t)=t^2$ with $J(\overline x,\overline u)=0$.
It is trivial to check that $(\overline x,\overline u)$ satisfies
the fractional necessary optimality conditions given by
Theorem~\ref{Opt:MainTheo}/Corollary~\ref{OPT:MainCor}.
\end{example}

Let us apply the fractional necessary conditions
to the above problem. The Hamiltonian is
$H=\left(tu-(\a+2)x\right)^2+\l u+\l t^2$.
The stationary condition \eqref{OPT:stacionary} implies that
for $t \ne 0$
$$
u(t)=\frac{\a+2}{t}x(t)-\frac{\l(t)}{2t^2}
$$
and hence
\begin{equation}
\label{emx1Ham}
H=-\frac{\l^2}{4t^2}+\frac{\a+2}{t}x\l+t^2\l,
\quad t \ne 0.
\end{equation}
Finally, \eqref{OPT:HamilSyst} gives
$$
\begin{cases}
\dot{x}(t)+{^C_0D^\a_t} x(t)=-\frac{\l}{2t^2}+\frac{\a+2}{t}x(t)+t^2,\\
-\dot{\l}(t)+{_tD^\a_1}\l(t)=\frac{\a+2}{t}\l(t),
\end{cases}
\quad
\begin{cases}
x(0)=0,\\
x(1)=\frac{2}{\Gamma(3+\a)}.
\end{cases}
$$
At this point, we encounter a fractional boundary value problem that needs
to be solved in order to reach the optimal solution. A handful of methods can be found
in the literature to solve this problem. Nevertheless, we use approximations
\eqref{expanMom} and \eqref{expanMomR}, up to order $N$, that have been introduced
in \cite{Atan2} and used in \cite{Jelicic,Pooseh2}. With our choice of approximation,
the fractional problem is transformed into a classical (integer order) boundary value problem:
$$
\begin{cases}
\dot{x}(t)=\left[\left(\frac{\a+2}{t}-At^{-\a}\right)x(t)
+\sum_{p=2}^N C_pt^{1-p-\a}V_p(t)-\frac{\l(t)}{2t^2}
+t^2\right]\frac{1}{1+Bt^{1-\a}}\\
\dot{V}_p(t)=(1-p)t^{p-2}x(t), \quad p=2,\ldots,N\\
\dot{\l}(t)=\left[\left(A(1-t)^{-\a}-\frac{\a+2}{t}\right)\l(t)
-\sum_{p=2}^N C_p(1-t)^{1-p-\a}W_p(t)\right]\frac{1}{1+B(1-t)^{1-\a}}\\
\dot{W}_p(t)=-(1-p)(1-t)^{p-2}\l(t), \quad p=2,\ldots,N
\end{cases}
$$
subject to the boundary conditions
$$
\begin{cases}
x(0)=0,\quad x(1)=\frac{2}{\Gamma(3+\a)},\\
V_p(0)=0, \quad p=2,\ldots,N,\\
W_p(1)=0, \quad p=2,\ldots,N.
\end{cases}
$$
The solutions are depicted in Figure~\ref{exm41FNCFig}
for $N=2$, $N=3$ and $\a=1/2$. Since the exact solution
for this problem is known, for each $N$ we compute
the approximation error by using the maximum norm.
Assume that $\overline x(t_i)$ are the approximated values
on the discrete time horizon $a=t_0,t_1,\ldots,t_n$.
Then the error is given by
$$
E = \max_{i}(|x(t_i)-\overline x(t_i)|).
$$
\begin{figure}
\begin{center}
\subfigure[$x(t), N=2$]{\label{xN2}\includegraphics[scale=0.31]{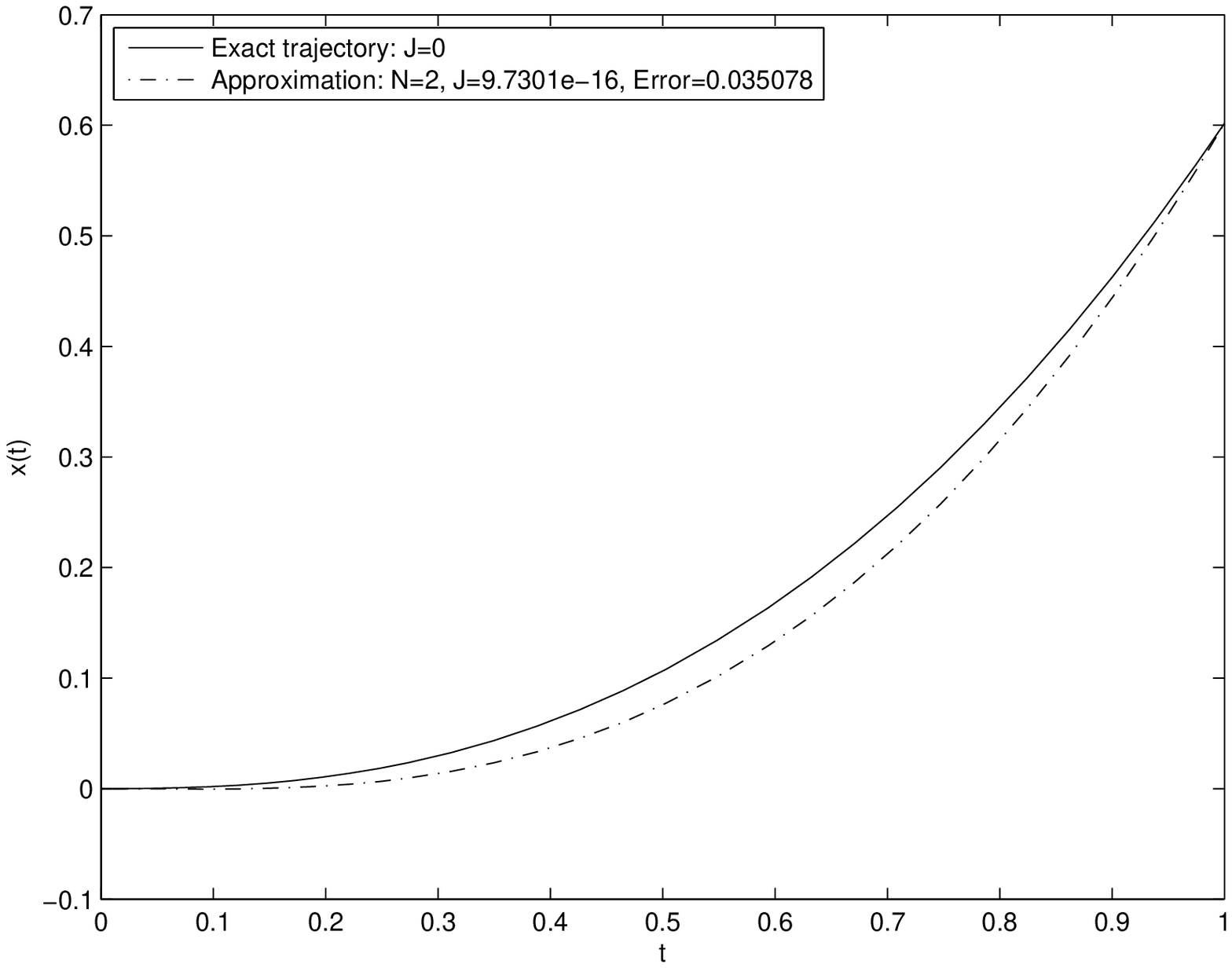}}
\subfigure[$u(t), N=2$]{\label{uN2}\includegraphics[scale=0.31]{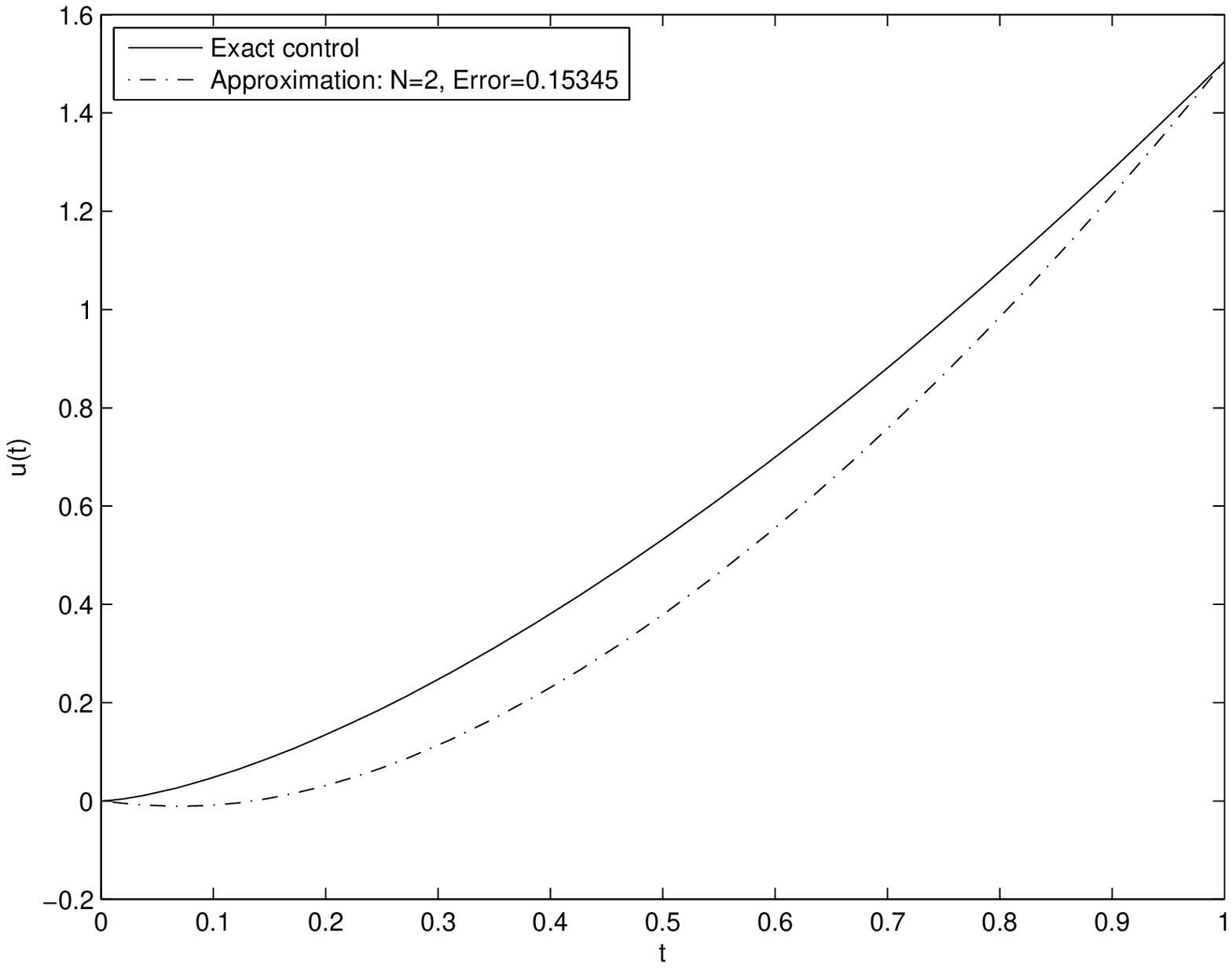}}
\subfigure[$x(t), N=3$]{\label{xN3}\includegraphics[scale=0.31]{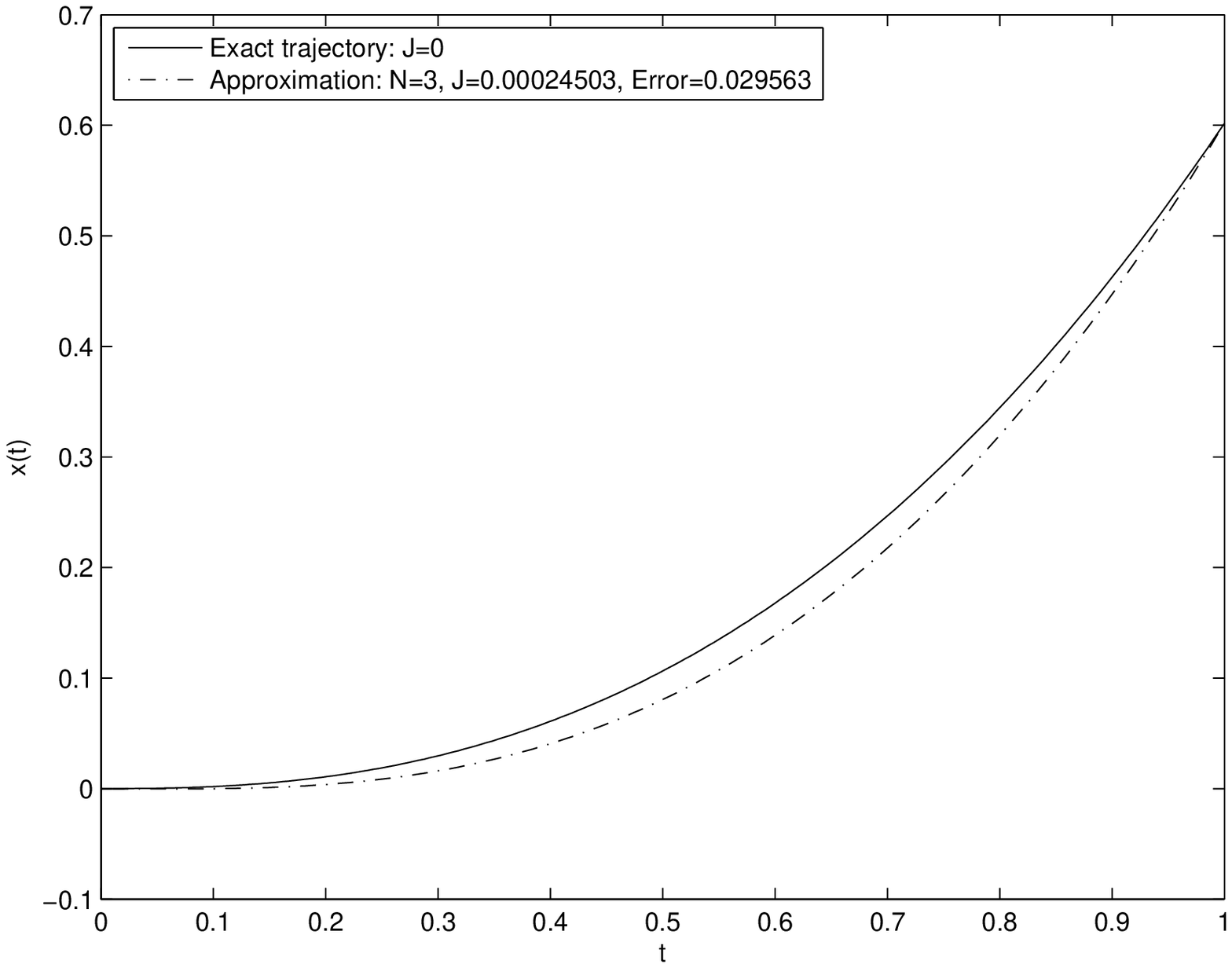}}
\subfigure[$u(t), N=3$]{\label{uN3}\includegraphics[scale=0.31]{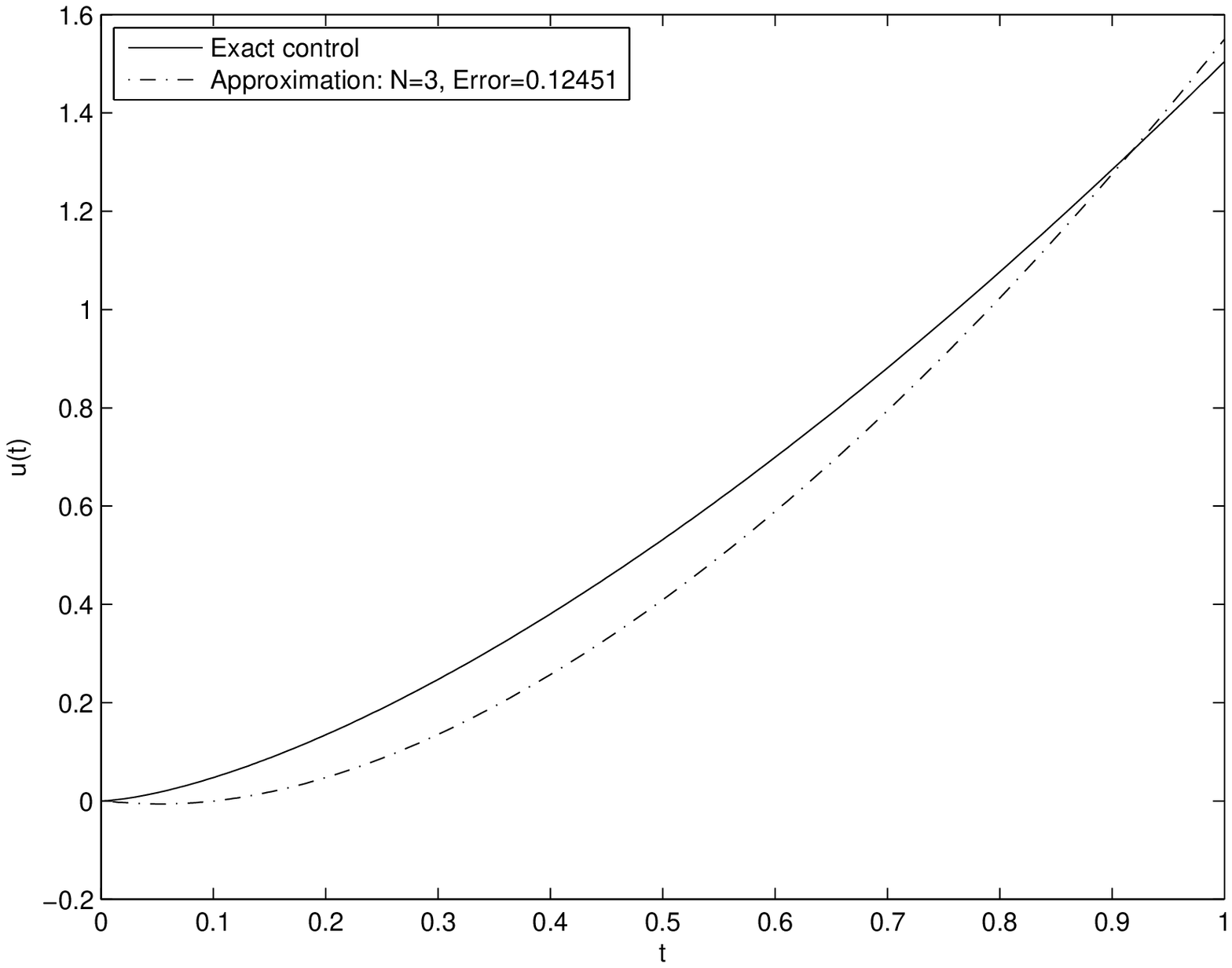}}
\end{center}
\caption{Exact solution (solid lines) for the problem in Example~\ref{exm41} with $\a = 1/2$
versus numerical solutions (dashed lines) obtained using approximations
\eqref{expanMom} and \eqref{expanMomR} up to order $N$
in the fractional necessary optimality conditions.}
\label{exm41FNCFig}
\end{figure}

Another approach is to approximate the original problem by using
\eqref{expanMom} for the fractional derivative. Following the procedure
discussed in Section~\ref{SecNecCond}, the problem of Example~\ref{exm41}
is approximated by
$$
\tilde{J}(x,u)=\int_0^1(tu-(\a+2)x)^2\,dt \longrightarrow \min
$$
subject to the control system
$$
\begin{cases}
\dot{x}(t)[1+B(\a,N)t^{1-\a}]+A(\a,N)t^{-\a}x(t)
-\sum_{p=2}^N C(\a,p)t^{1-p-\a}V_p(t)=u(t)+t^2\\
\dot{V}_p(t)=(1-p)t^{p-2}x(t)
\end{cases}
$$
and boundary conditions
$$
x(0)=0, \quad x(1)=\frac{2}{\Gamma(3+\a)},
\quad V_p(0)=0, \quad p=2,3,\ldots,N.
$$
The Hamiltonian system for this classical optimal control problem is
\begin{equation*}
\begin{split}
H=\left(tu-(\a+2)x\right)^2
&+\frac{\l_1(-A(\a,N)t^{-\a}x
+\sum_{p=2}^N C(\a,p)t^{1-p-\a}V_p+u+t^2)}{1+B(\a,N)t^{1-\a}}\\
&+\sum_{p=2}^N (1-p)t^{p-2}\l_p x.
\end{split}
\end{equation*}
Using the stationary condition $\frac{\partial H}{\partial u}=0$, we have
$$
u(t)=\frac{\a+2}{t}x(t)-\frac{\l_1(t)}{2t^2(1+B(\a,N)t^{1-\a})}
\quad \text{for } t \ne 0.
$$
Finally, the Hamiltonian becomes
\begin{equation}
\label{exmHamiltonian}
H=\phi_0\l_1^2+\phi_1x\l_1+\sum_{p=2}^N \phi_pV_p\l_1
+\phi_{N+1}\l_1+\sum_{p=2}^N(1-p)t^{p-2}x\l_p, \quad t \ne 0,
\end{equation}
where
\begin{equation}
\label{eq:phi0:phi1}
\phi_0(t)=\frac{-1}{4t^2(1+B(\a,N)t^{1-\a})^2},
\quad \phi_1(t)=\frac{\a+2-A(\a,N)t^{1-\a}}{t(1+B(\a,N)t^{1-\a})},
\end{equation}
and
\begin{equation}
\label{eq:phi2:phi3}
\phi_p(t)=\frac{C(\a,p)t^{1-p-\a}}{1+B(\a,N)t^{1-\a}},
\quad \phi_{N+1}(t)=\frac{t^2}{1+B(\a,N)t^{1-\a}}.
\end{equation}
The Hamiltonian system
$\dot{\mathbf{x}}=\frac{\partial H}{\partial \mathbf{\lambda}}$,
$\dot{\bm{\l}}=-\frac{\partial H}{\partial \mathbf{x}}$,
gives
$$
\begin{cases}
\dot{x}(t)=2\phi_0(t)\l_1(t)+\phi_1(t)x(t)+\sum_{p=2}^N \phi_p(t)V_p(t)+\phi_{N+1}(t)\\
\dot{V_p}=(1-p)t^{p-2}x(t), \quad p=2,\ldots,N\\
\dot{\lambda_1}=-\phi_1(t)\l_1(t)+\sum_{p=2}^N(p-1)t^{p-2}\l_p\\
\dot{\lambda_p}=-\phi_p(t)\l_1(t), \quad p=2,\ldots,N
\end{cases}
$$
subject to the boundary conditions
$$
\begin{cases}
x(0)=0\\
V_p(0)=0, \quad p=2,\ldots,N
\end{cases}
\qquad
\begin{cases}
x(1)=\frac{2}{\Gamma(3+\a)}\\
\l_p(1)=0, \quad p=2,\ldots,N.
\end{cases}
$$
This two-point boundary value problem was solved using Matlab's \textsf{bvp4c}
built-in function for $N=2$ and $N=3$. The results are depicted in Figure~\ref{exm41Fig}.
\begin{figure}
\begin{center}
\subfigure[$x(t), N=2$]{\includegraphics[scale=0.42]{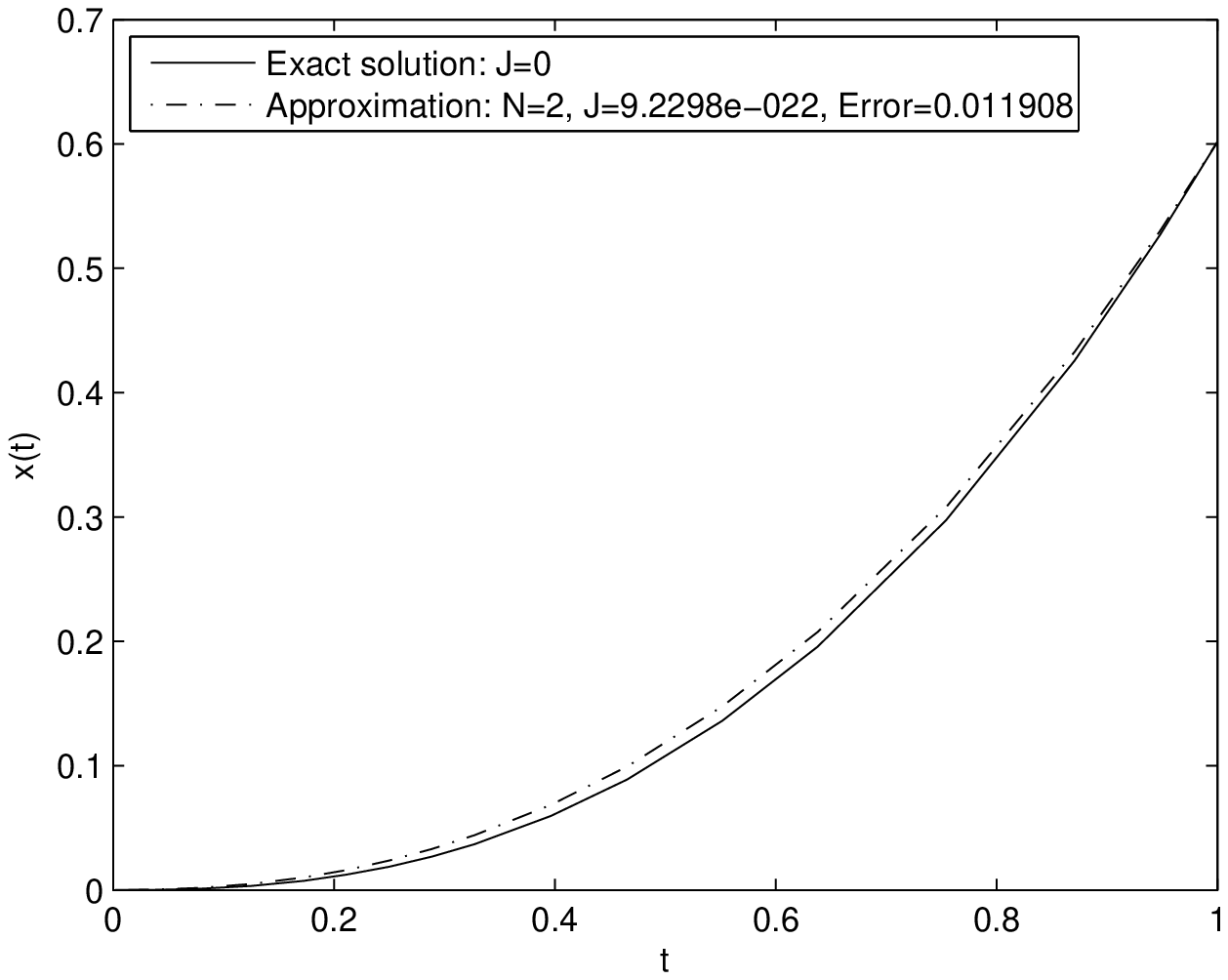}}
\subfigure[$u(t), N=2$]{\includegraphics[scale=0.42]{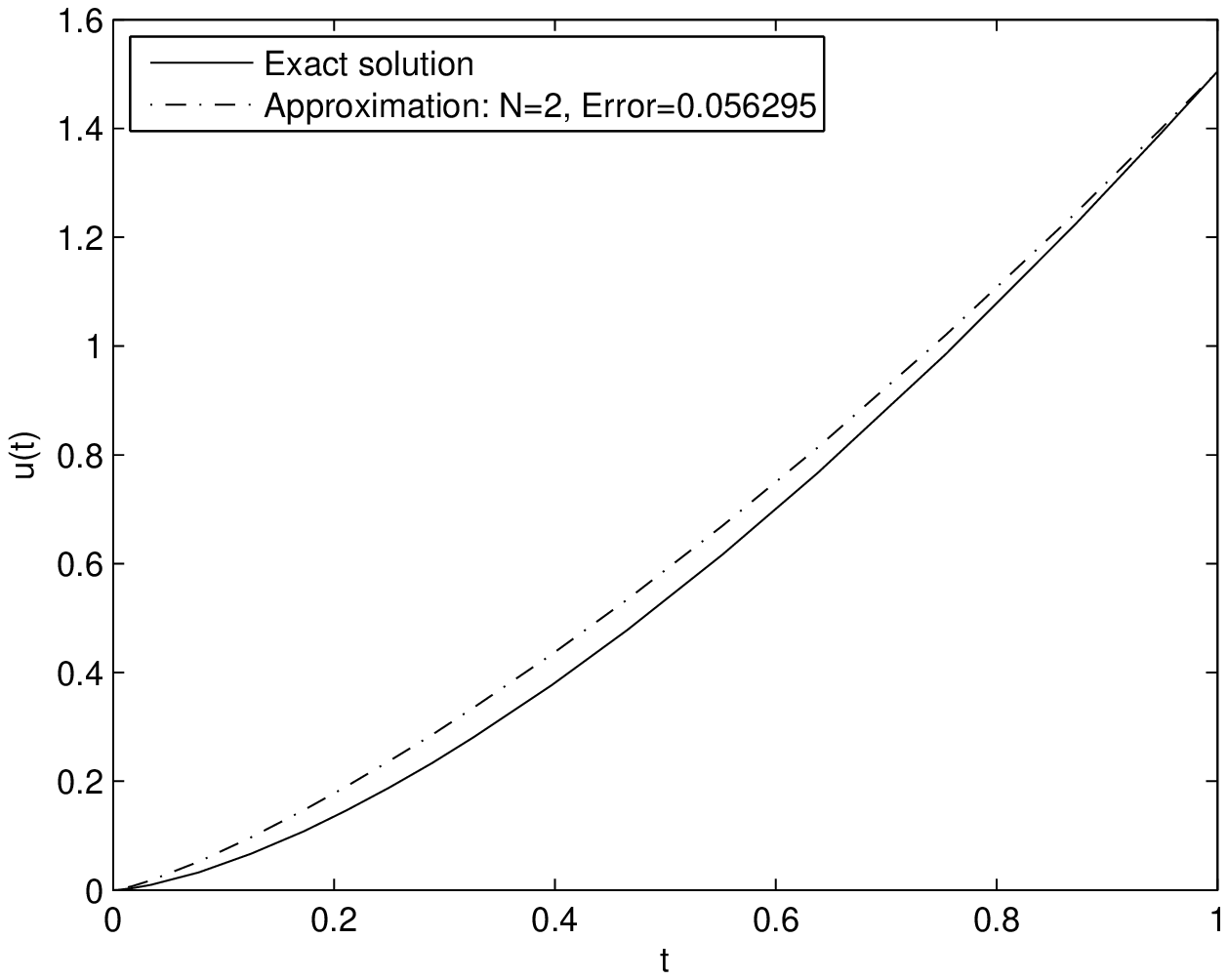}}
\subfigure[$x(t), N=3$]{\includegraphics[scale=0.42]{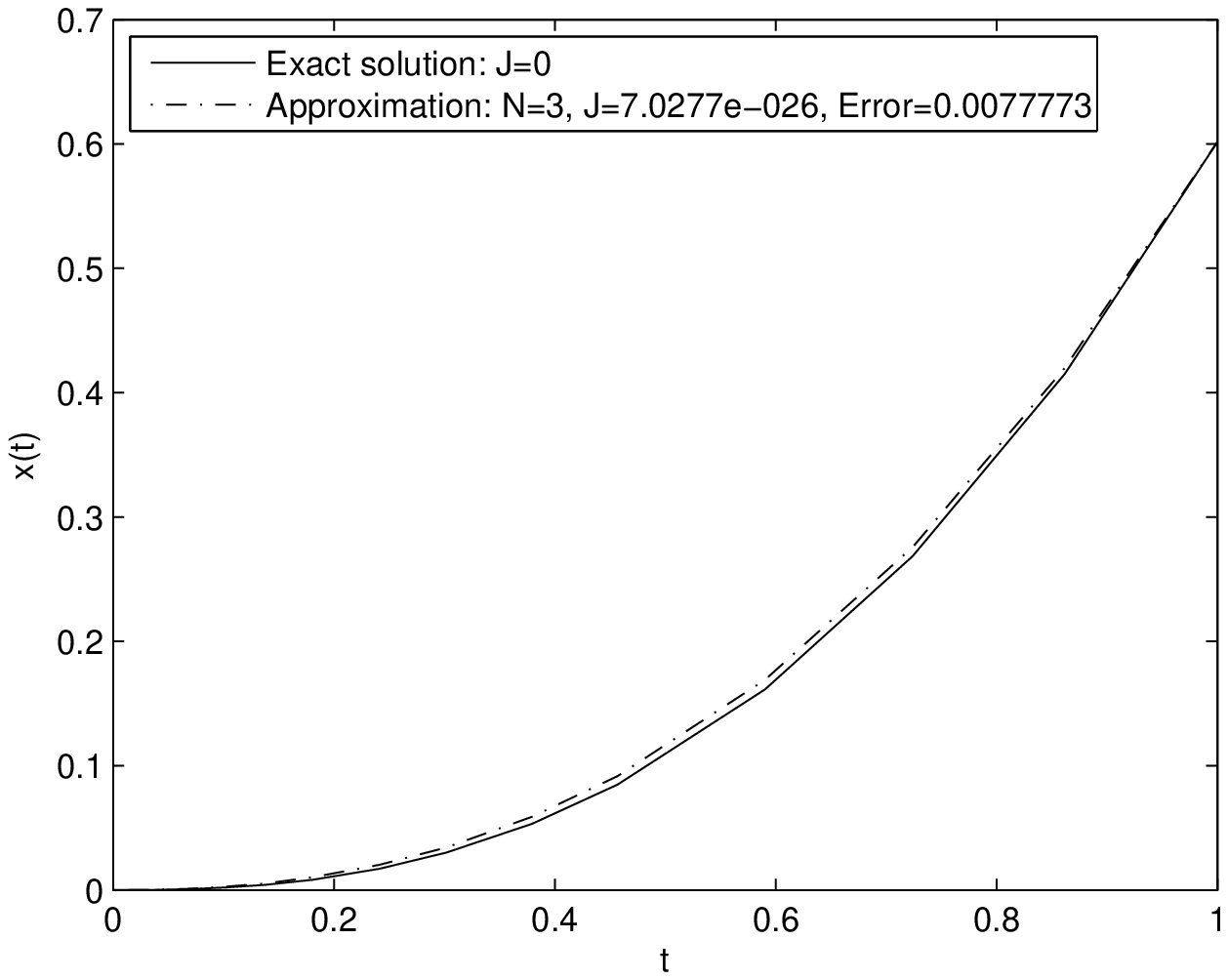}}
\subfigure[$u(t), N=3$]{\includegraphics[scale=0.42]{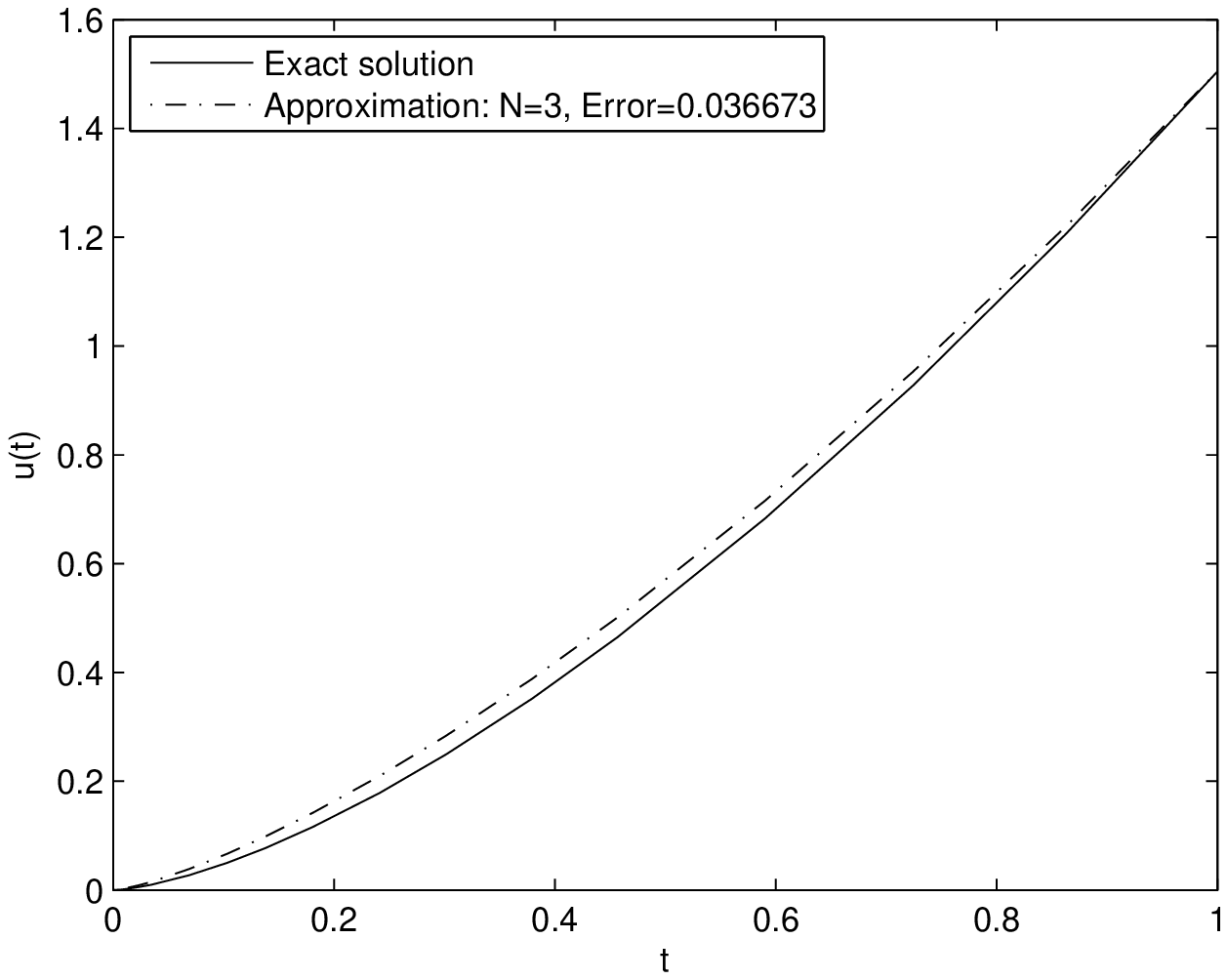}}
\end{center}
\caption{Exact solution (solid lines) for the problem in Example~\ref{exm41} with $\a = 1/2$
versus numerical solutions (dashed lines) obtained by approximating the fractional order
optimal control problem using \eqref{expanMom} up to order $N$ and then
solving the classical necessary optimality conditions with Matlab's \textsf{bvp4c}
built-in function.} \label{exm41Fig}
\end{figure}


\subsection{Free final time}
\label{sub:sec:freeft}

The two numerical methods discussed in Section~\ref{sub:sec:fft}
are now employed to solve a fractional order optimal control
problem with free final time $T$.

\begin{example}
\label{exm42}
Find an optimal triplet $(x(\cdot),u(\cdot),T)$ that minimizes
$$
J(x,u)=\int_0^T(tu-(\a+2)x)^2\,dt
$$
subject to the control system
$$
\dot{x}(t)+{^C_0D^\a_t} x(t)=u(t)+t^2\\
$$
and boundary conditions
$$
x(0)=0, \quad x(T)=1.
$$
An exact solution to this problem is not known
and we apply the two numerical procedures
already used with respect to the fixed final time
problem in Example~\ref{exm41}.
\end{example}

We begin by using the fractional necessary optimality conditions that,
after approximating the fractional terms, result in
$$
\begin{cases}
\dot{x}(t)=\left[\left(\frac{\a+2}{t}-At^{-\a}\right)x(t)
+\sum_{p=2}^N C_pt^{1-p-\a}V_p(t)
-\frac{\l(t)}{2t^2}+t^2\right]\frac{1}{1+Bt^{1-\a}}\\
\dot{V}_p(t)=(1-p)t^{p-2}x(t), \quad p=2,\ldots,N\\
\dot{\l}(t)=\left[\left(A(1-t)^{-\a}-\frac{\a+2}{t}\right)\l(t)
-\sum_{p=2}^N C_p(1-t)^{1-p-\a}W_p(t)\right]\frac{1}{1+B(1-t)^{1-\a}}\\
\dot{W}_p(t)=-(1-p)(1-t)^{p-2}\l(t), \quad p=2,\ldots,N
\end{cases}
$$
subject to the boundary conditions
$$
\begin{cases}
x(0)=0,\quad x(T)=1,\\
V_p(0)=0, \quad p=2,\ldots,N,\\
W_p(T)=0, \quad p=2,\ldots,N.
\end{cases}
$$
The only difference here with respect to Example~\ref{exm41}
is that there is an extra unknown, the terminal time $T$.
The boundary condition for this new unknown is chosen appropriately from the
transversality conditions discussed in Corollary~\ref{OPT:MainCor}, i.e.,
$$
[H(t,x,u,\l)-\l(t)\LCD x(t)+\dot{x}(t)\RIT\l(t)]_{t=T}=0,
$$
where $H$ is given as in \eqref{emx1Ham}.
Since we require $\l$ to be continuous,
$\RIT\l(t)|_{t=T}=0$ (cf. \cite[pag.~46]{Miller}) and so  $\l(T)=0$.
One possible way to proceed consists in translating the problem into the interval
$[0,1]$ by the change of variable $t=Ts$ \cite{Avvakumov}. In this setting,
either we add $T$ to the problem as a new state variable with dynamics $\dot{T}(s)=0$,
or we treat it as a parameter. We use the latter, to get the following
parametric boundary value problem:
$$
\begin{cases}
\dot{x}(s)=\frac{\left[\left(\frac{\a+2}{Ts}-A(Ts)^{-\a}\right)x(s)
+\sum_{p=2}^N C_p(Ts)^{1-p-\a}V_p(s)
-\frac{\l(s)}{2(Ts)^2}+(Ts)^2\right] T}{1+B(Ts)^{1-\a}},\\
\dot{V}_p(s)=T(1-p)(Ts)^{p-2}x(s), \quad p=2,\ldots,N,\\
\dot{\l}(s)=\frac{\left[\left(A(1-Ts)^{-\a}-\frac{\a+2}{Ts}\right)\l(s)
-\sum_{p=2}^N C_p(1-Ts)^{1-p-\a}W_p(s)\right] T}{1+B(1-Ts)^{1-\a}},\\
\dot{W}_p(s)=-T(1-p)(1-Ts)^{p-2}\l(s), \quad p=2,\ldots,N,
\end{cases}
$$
subject to the boundary conditions
$$
\begin{cases}
x(0)=0,\\
V_p(0)=0, \quad p=2,\ldots,N,\\
W_p(1)=0, \quad p=2,\ldots,N,
\end{cases}
\qquad
\begin{cases}
x(1)=1,\\
\l(1)=0.
\end{cases}
$$
This parametric boundary value problem is solved for $N=2$ and $\a=0.5$
with Matlab's \textsf{bvp4c} function. The result is shown
in Figure~\ref{exm42fig} (dashed lines).
\begin{figure}
\begin{center}
\subfigure[$x(t), N=2$]{\includegraphics[scale=0.31]{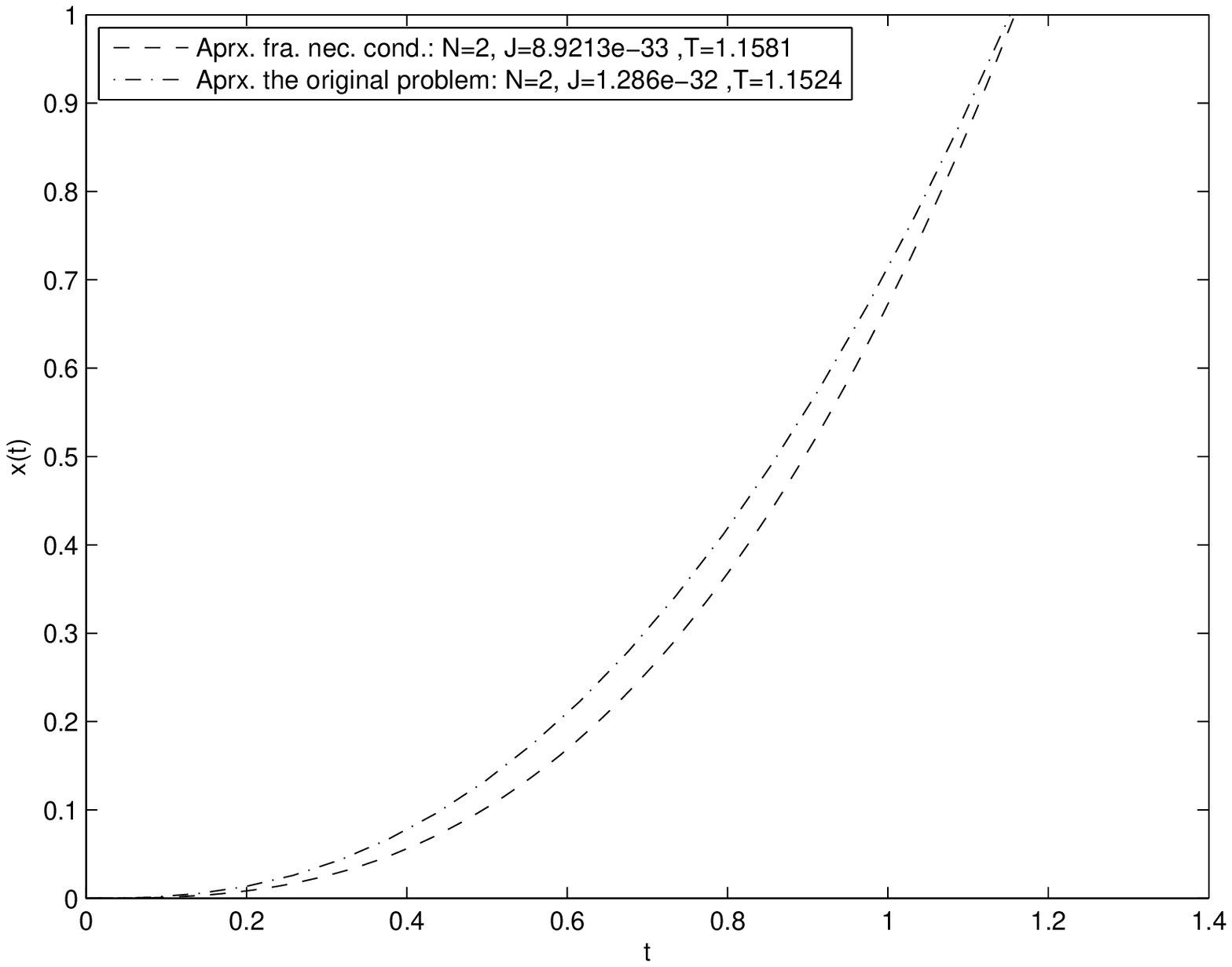}}
\subfigure[$u(t), N=2$]{\includegraphics[scale=0.31]{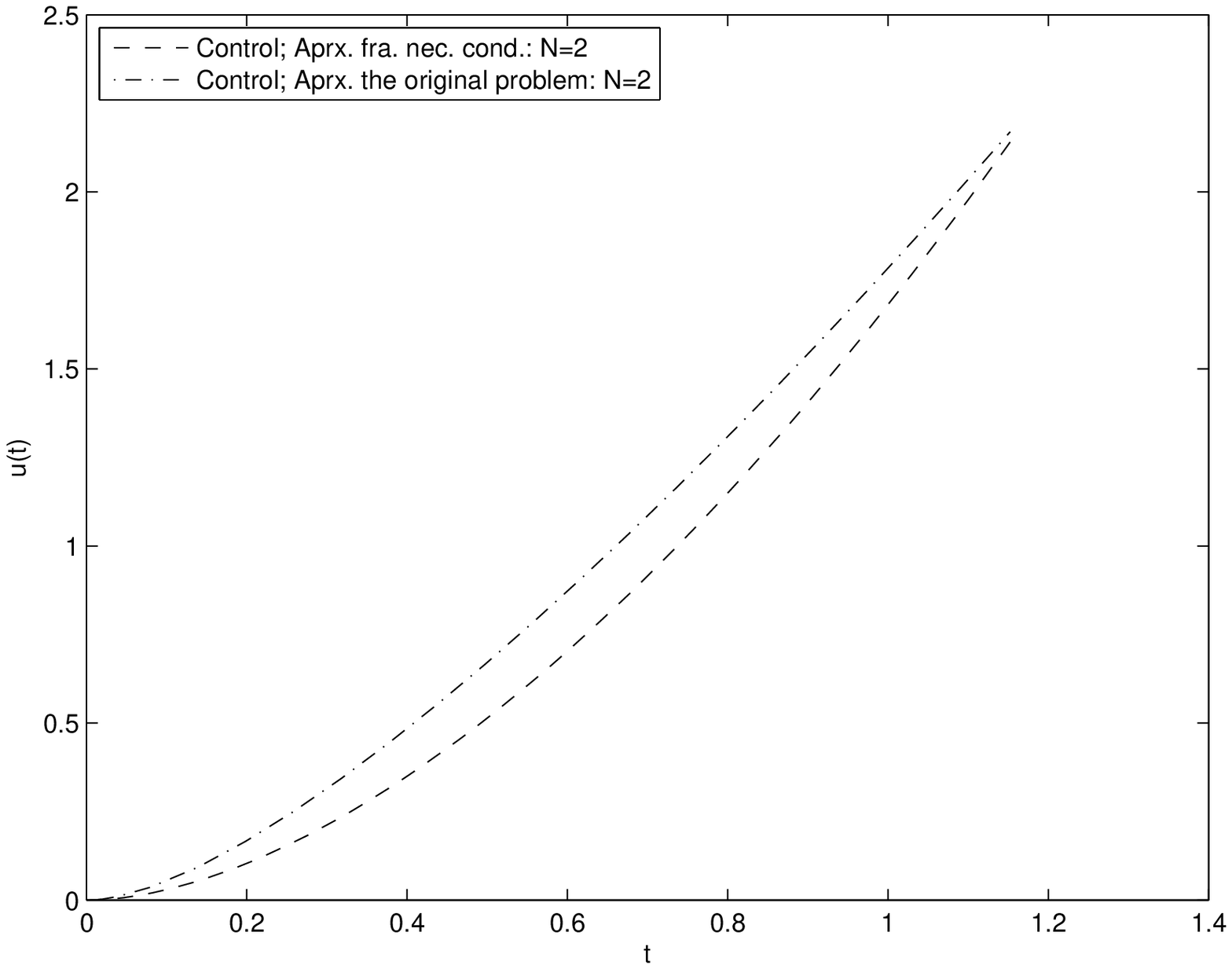}}
\end{center}
\caption{Numerical solutions to the free final time problem of Example~\ref{exm42}
with $\alpha = 1/2$, using fractional necessary optimality conditions (dashed lines) and
approximation of the problem to an integer order optimal control problem (dash-dotted lines).}
\label{exm42fig}
\end{figure}

We also solve Example~\ref{exm42} with $\alpha = 1/2$ by directly transforming it
into an integer order optimal control problem with free final time. As is well known
in the classical theory of optimal control, the Hamiltonian must vanish at the terminal point
when the final time is free, i.e., one has $H|_{t=T}=0$ with $H$ given
by \eqref{exmHamiltonian} \cite{Kirk}. For $N=2$, the necessary optimality conditions
give the following two point boundary value problem:
$$
\begin{cases}
\dot{x}(t)=2\phi_0(t)\l_1(t)+\phi_1(t)x(t)+\phi_2(t)V_2(t)+\phi_{3}(t)\\
\dot{V_2}=-x(t)\\
\dot{\lambda_1}=-\phi_1(t)\l_1(t)+x(t)\\
\dot{\lambda_2}=-\phi_2(t)\l_1(t),
\end{cases}
$$
where $\phi_0(t)$ and $\phi_1(t)$ are given by
\eqref{eq:phi0:phi1} and $\phi_2(t)$ and $\phi_{3}(t)$
by \eqref{eq:phi2:phi3} with $p = N = 2$.
The trajectory $x$ and corresponding $u$
are shown in Figure~\ref{exm42fig} (dash-dotted lines).


\section*{Acknowledgments}

Work supported by {\it FEDER} funds through {\it COMPETE}
--- Operational Programme Factors of Competitiveness
(``Programa Operacional Factores de Competitividade'')
and by Portuguese funds through the {\it Center for Research
and Development in Mathematics and Applications} (University of Aveiro)
and the Portuguese Foundation for Science and Technology
(``FCT--Funda\c{c}\~{a}o para a Ci\^{e}ncia e a Tecnologia''),
within project PEst-C/MAT/UI4106/2011 with COMPETE
number FCOMP-01-0124-FEDER-022690. Pooseh was also
supported by the FCT Ph.D. fellowship SFRH/BD/33761/2009.

The authors would like to thank two anonymous referees
for their careful reading of the manuscript
and for suggesting several useful changes;
and to Dr. Ryan Loxton for suggestions for improving the English.



\medskip

Received December 2012; revised January and February 2013.

\medskip


\end{document}